\documentclass[preprint,12pt]{elsarticle}

\usepackage{amssymb}
\usepackage{graphicx}
\usepackage{mathptmx}      

\usepackage{amssymb,amsfonts,amsmath,amscd,latexsym}
\usepackage[english]{babel}
\usepackage{geometry}
\usepackage{latexsym}
\usepackage{amssymb}
\usepackage[all]{xy}
\usepackage{amssymb}
\usepackage{latexsym}
\usepackage{graphicx}
\usepackage{appendix}
\usepackage[active]{srcltx}
\usepackage{srcltx}
\usepackage{xypic}
\usepackage{amsthm}

\newdefinition{defn}{Definition}[section]
\newdefinition{nim}[defn]{}
\newdefinition{rem}[defn]{Remark}
\newdefinition{ex}[defn]{Example}
\newtheorem{thm}{Theorem}[section]

\newtheorem{lem}[thm]{Lemma}
\newtheorem{prop}[thm]{Proposition}

\newcommand{\Tr}{\mathrm{Tr}}

\newcommand{\Aut}{\mathrm{Aut}}

\newcommand{\Ker}{\mathrm{Ker}}

\newcommand{\Br}{\mathrm{Br}}

\begin{document}

\begin{frontmatter}



\title{The extended Brauer quotient of $N$-interior $G$-algebras}

\author{Tiberiu Cocone\c t\corref{cor1}}
\ead{tiberiu.coconet@math.ubbcluj.ro}
\address{Faculty of Economics and Business Administration, Babe\c s-Bolyai University, Str. Teodor Mihali, nr.58-60 , 400591 Cluj-Napoca, Romania}
\author{Constantin-Cosmin Todea}
\ead{Constantin.Todea@math.utcluj.ro}
\address{Department of Mathematics, Technical University of Cluj-Napoca, Str. G. Baritiu 25,
 Cluj-Napoca 400027, Romania}

\cortext[cor1]{Corresponding author} \fntext[secondadr]{The first author acknowledges the  support of  the Romanian
PN-II-ID-PCE-2012-4-0100 project}
\begin{abstract} In this note we generalize  the extended Brauer quotient defined by L. Puig and Y. Zhou in \cite[Section 3]{PuZh} to the case of $N$-interior $G$-algebras, where $N$ is a normal subgroup of a
finite group $G$. We use this extended Brauer quotient on
permutation algebras  to establish some correspondences that
generalize known results that hold in the group algebra case.

\end{abstract}


\end{frontmatter}


\section{Introduction}\label{sec1}
Throughout this paper  $N$ is a normal subgroup of a finite group
$G$, $p$ is a prime and  $k$ is the residue field of a discrete valuation ring $\mathcal{O}$, with characteristic $p$. We assume that $k$ is algebraically closed. All modules,
considered as left modules, and algebras are finitely
generated.  Recall that an   $N$-interior $G$-algebra $A$ is  a
$k$-algebra endowed with two group homomorphisms
$$\alpha:N\rightarrow A^{\times}$$ and
$$\phi:G\rightarrow \Aut_k(A),$$ where
$A^{\times}$ is the subgroup of invertible elements of $A$ and
$\Aut_k(A)$ is the group of $k$-algebra automorphisms of $A$. For
any $g\in G, n\in N, a\in A$ we set ${}^ga=\phi(g)(a),
{}^gn=gng^{-1}$. Moreover, we have  ${}^na=\alpha(n)a\alpha(n)^{-1}$.
Note that since $A$ is $N$-interior it carries   a $kN-kN$-bimodule
structure. The algebra $A$ serves for the construction of an
interior $G$-algebra
\[C:=A\otimes_NG,\] where for $a,b\in A$ and $x,y\in G,$   the product is given by
\[(a\otimes x)(b\otimes y)=a~{}^xb\otimes xy.\]

The extended Brauer quotient  for interior $G$-algebras was defined
and analyzed in \cite{PuZh} and in \cite{PuZhII}. It turns out that
this extended Brauer quotient is an important tool for generalizing
some results regarding basic Morita equivalences and basic Rickard
equivalences between  blocks of finite groups. Basic Morita and Rickard equivalences are intensively studied in \cite{PuLo}.
In \cite{PuZh} the
definition of the extended Brauer quotient is stated using any
subgroup of $\Aut(P),$ where $P$ is a $p$-subgroup of $G.$ For
technical reasons, that are clarified later  on this paper, we
restrict ourselves to the set
$$K:=\{\varphi\in\Aut(P)\mid \varphi(u)u^{-1}\in N,\forall u\in P\},$$
which is actually  a subgroup of $\Aut(P)$. If $\varphi\in K$ we
denote by $\Delta_{\varphi}:P\rightarrow P\times P$ the
$\varphi$-twisted diagonal, that is the group homomorphism defined
by $\Delta_{\varphi}(u)=(\varphi(u),u)$ for any $u\in P$.  For any
subgroup $T$ of $K$  the $T$-normalizer $N_G^T(P)$ of $P$ in $G$ is
the inverse image of $T$ through the canonical homomorphism
$N_G(P)\rightarrow \Aut(P)$.

 In Section \ref{sec2} we make necessary constructions for defining the
extended Brauer quotients of an $N$-interior $G$-algebra. Note that
Definition \ref{defnextBrauerquot}, although is seems similar to
that of \cite{PuZh},  yields an additional action of $N_G(P)$ on the
extended Brauer quotient. The mentioned action  is an important
ingredient used in  Section \ref{sec3}, since  we obtain two
analogous  correspondences similar to \cite[Proposition 3.3]{PuZh}.
Explicitly these are Theorem \ref{thmprop33} and Theorem \ref{thm}.
Also for a subalgebra of $C,$ we obtain in  Proposition \ref{thmAhat} a
relationship
between the extended Brauer quotient defined in \cite{PuZh}
and the one constructed in Section \ref{sec2}.

Our general assumptions and notations are standard. We refer the
readers to \cite{The} and \cite{Pu} for
Puig's theory of $G$-algebras and pointed groups.

\section{The Brauer quotient of $N$-interior $G$-algebras }\label{sec2}
We begin with some useful propositions. With the above notations we
have:
\begin{prop}\label{propADeltafhi} The $k$-algebra $A$ is a $k\Delta_{\varphi}(P)$-module via the action $$(\varphi(u),u)a:=\varphi(u)u^{-1}~{}^ua,$$
for any $u\in P, a\in A$.
\end{prop}
\begin{proof} Let $u_1,u_2\in P, a\in A$. The following equalities are enough to prove the proposition
$$[(\varphi(u_1),u_1)(\varphi(u_2),u_2)]a=\varphi(u_1u_2)u_2^{-1}u_1^{-1}({}^{u_1u_2}a);$$
$$(\varphi(u_1),u_1)[(\varphi(u_2),u_2)a]=(\varphi(u_1),u_1)[\varphi(u_2)u_2^{-1}~{}^{u_2}a]=\varphi(u_1)u_1^{-1}[{}^{u_1}(\varphi(u_2)u_2^{-1}~{}^{u_2}a)]=$$
$$\varphi(u_1)u_1^{-1}u_1\varphi(u_2)u_1^{-1}[{}^{u_1}(u_2^{-1})]({}^{u_1u_2}a)=\varphi(u_1u_2)u_2^{-1}u_1^{-1}({}^{u_1u_2}a).$$
\end{proof}
We denote the \emph{$\varphi$-normalizer of $P$ in $A$} by
$$N_A^{\varphi}(P):=\{a\in A\mid (\varphi(u),u)a=a, \forall u\in P\}.$$
Proposition \ref{propADeltafhi} implies that this set is the
$k$-subspace of $\Delta_{\varphi}(P)$-fixed elements, also denoted
$A^{\Delta_{\varphi}(P)}$. As usual, we may consider the  Brauer
map, which is the canonical surjection
$$\Br_{\Delta_{\varphi}(P)}^A:N_A^{\varphi}(P)\rightarrow N_A^{\varphi}(P)/(\sum_{R<P}\Tr_{\Delta_{\varphi}(R)}^{\Delta_{\varphi}(P)}(N_A^{\varphi}(R))),$$
where $\Tr_{\Delta_{\varphi}(R)}^{\Delta_{\varphi}(P)}$ denotes the
\emph{relative trace map}, see \cite[2.5]{PuZh} for more details. We
define an external direct sum
$$N_A^K(P):=\bigoplus_{\varphi\in K}N_A^{\varphi}(P).$$

\begin{prop}\label{propdefnNAK} With the above notations the following hold.
\begin{itemize}
\item[i)] For any  $\varphi,\psi\in K$ we have  $N_A^{\psi}(P)N_A^{\varphi}(P)\subseteq
N_A^{\psi\circ\varphi}(P)$;
\item[ii)] The group $N_G(P)$ acts on  $K$ in the following way. If  $(x,\varphi)\in N_G(P)\times K$ then $(x,\varphi)\mapsto {}^x\varphi$, where ${}^x\varphi:P\rightarrow P$ with
$${}^x\varphi(u):=x\varphi(x^{-1}ux)x^{-1}$$ for any $u\in P$;
\item[iii)] $N_N^K(P)$ is a normal subgroup of $N_G(P)$, hence the external direct sum  $N_A^K(P)$ is an $N_N^K(P)$-interior $N_G(P)$-algebra.
\end{itemize}
\end{prop}

\begin{proof}  i) Let $a\in N_A^{\psi}(P)$, $b\in N_A^{\varphi}(P)$ and $u\in P$. Then $\psi(u)u^{-1}~{}^ua=a$ and $\varphi(u)u^{-1}~{}^ub=b$.
Then \begin{align*}(\psi\circ
\varphi)(u)u^{-1}[{}^u(ab)]=&\psi(\varphi(u))u^{-1}({}^ua~{}^ub)\\=&\psi(\varphi(u))\varphi(u)^{-1}\varphi(u)u^{-1}({}^ua)u\varphi(u^{-1})\varphi(u)u^{-1}({}^ub)\\
=&\psi(\varphi(u))\varphi(u^{-1})[{}^{\varphi(u)u^{-1}}({}^ua)]b=\psi(\varphi(u))\varphi(u)^{-1}{}(^{\varphi(u)}a)b=ab.\end{align*}
ii) If $x\in N_G(P)$ then ${}^x\varphi$ is clearly in $\Aut(P).$
Since for any $u\in P$ we have $$x^{-1}u^{-1}x\varphi(x^{-1}ux)\in
N,$$ which is equivalent to $u^{-1}~{}^x\varphi(u)\in N,$ we obtain that ${}^x\varphi$
is in $K.$

iii) Statement i) gives a $k$-algebra structure on $N_A^K(P).$ So it
is enough to prove that for any $x\in N_G(P)$ and $a\in
N_A^{\varphi}(P)$ the element  ${}^xa$ belongs to
$N_A^{{}^x\varphi}(P),$ since this yields that $N_A^K(P)$ is an
$N_N^K(P)$-interior $N_G(P)$-algebra. For any $u\in P$  there is
$u'\in P$ such that $x^{-1}ux=u'$. We have
\begin{align*}{}^x\varphi(u)u^{-1}~{}^u({}^xa)&=x\varphi(x^{-1}ux)x^{-1}u^{-1}xx^{-1}({}^{ux}a)\\
&=x\varphi(u')u'^{-1}x^{-1}({}^{xu'}a)={}^x(\varphi(u')u'^{-1}~{}^{u'}a)={}^xa.\end{align*}
\end{proof}

\begin{lem}\label{lem} Let $R,R'\leq P$ and $\varphi,\varphi'\in K$. If
 $a=\Tr_{\Delta_{\varphi}(R)}^{\Delta_{\varphi}(P)}(c)$ and $a'=\Tr_{\Delta_{\varphi'}(R')}^{\Delta_{\varphi'}(P)}(c')$ then $aa'=\Tr_{\Delta_{\varphi\circ\varphi'}(R')}^{\Delta_{\varphi\circ\varphi'}(P)}(ac')=\Tr_{\Delta_{\varphi\circ\varphi'}(R')}^{\Delta_{\varphi\circ\varphi'}(P)}(ca')$.
\end{lem}
\begin{proof}
We know that
$$a'=\sum_{(\varphi'(u),u)\in \Delta_{\varphi'}(P)/\Delta_{\varphi'}(R')}\varphi'(u)u^{-1}~{}^u(c').$$
The following equalities hold
\begin{align*}aa'&=\sum_{(\varphi'(u),u)\in \Delta_{\varphi'}(P)/\Delta_{\varphi'}(R')}a\varphi'(u)u^{-1}~{}^u(c')\\&=
\sum_{(\varphi'(u),u)\in \Delta_{\varphi'}(P)/\Delta_{\varphi'}(R')}\varphi(\varphi'(u))\varphi'(u^{-1})({}^{\varphi'(u)}a)\varphi'(u)u^{-1}~{}^u(c')\\
&=\sum_{(\varphi'(u),u)\in \Delta_{\varphi'}(P)/\Delta_{\varphi'}(R')}\varphi(\varphi'(u))u^{-1}u\varphi'(u^{-1})({}^{\varphi'(u)}a)\varphi'(u)u^{-1}~{^u}(c')\\
&=\sum_{(\varphi'(u),u)\in \Delta_{\varphi'}(P)/\Delta_{\varphi'}(R')}\varphi(\varphi'(u))u^{-1}[{}^{u\varphi'(u^{-1})}({}^{\varphi'(u)}a)]~{}^u(c')\\
&=\sum_{(\varphi(\varphi'(u)),u)\in \Delta_{\varphi\circ\varphi'}(P)/\Delta_{\varphi\circ\varphi'}(R')}\varphi(\varphi'(u))u^{-1}~{}^u(ac')\\
&=\Tr_{\Delta_{\varphi\circ\varphi'}(R)}^{\Delta_{\varphi\circ\varphi'}(P)}(ac').
\end{align*}
Similar computations give  the second equality from the conclusion
of the lemma.
\end{proof}

Note that if $R\neq P$ or $R'\neq P$  Lemma \ref{lem} implies that
$aa'\in \Ker (\Br_{\Delta_{\varphi\circ\varphi'}(P)}^A)$, hence
 $\bigoplus_{\varphi\in K}\Ker (\Br_{\Delta_{\varphi}(P)}^A)$ is a two-sided ideal of $N_A^K(P)$. This allows us  to consider the following definition.
\begin{defn}\label{defnextBrauerquot} The \emph{extended Brauer quotient}  is the following  $N_N^K(P)$-interior $N_G(P)$-algebra
$$\overline{N}_A^K(P):=N_A^K(P)/(\bigoplus_{\varphi\in K}\Ker (\Br_{\Delta_{\varphi}(P)}^A).$$
\end{defn}

The following two results contain some useful properties of the
extended Brauer quotient for some $N$-interior $G$-algebras. For
more details regarding $N$-interior $G$-algebras see \cite{Pu}. The
proof of the next proposition is similar to \cite[Proposition
3.4]{PuZh} and is left as an exercise for the reader.
\begin{prop}\label{propcompmorhism} Let $A,B$ be two $N$-interior $G$-algebras and $f:A\rightarrow B$ a homomorphism of $N$-interior $G$-algebras.
Then $f$ induces two  homomorphisms of $N_N^K(P)$-interior
$N_G(P)$-algebras $N_f^K:N_A^K(P)\rightarrow N_B^K(P)$ and
$\overline{N}_f^K:\overline{N}_A^K(P)\rightarrow\overline{N}_B^K(P)$.
Moreover if $f$ is an embedding then $N_f^K$ and $\overline{N}_f^K$
are embeddings too.
\end{prop}

Let $K'$ denote the group of interior automorphisms of $P$ induced
by conjugation with the elements of $N$. Clearly $K'$  is a subgroup
of $K$.
\begin{prop}\label{propgroupalgextbr} Let $A=kN$ as an $N$-interior $G$-algebra. There exists an isomorphism of $N_N(P)$-interior
 $N_G(P)$-algebras
$$\rho^{K'}:kN_N(P)\rightarrow \overline{N}_A^{K'}(P).$$
\end{prop}

\begin{proof}   It is known that the groups $K'$ and $N_N(P)/C_N(P)$ are isomorphic. Definition \ref{defnextBrauerquot} assures
the existence of an algebra  homomorphism
\[kN_N(P)\to \bar{N}_A^{K'}(P),\] which is actually a homomorphism
of  $K'$-graded algebras since there is a bijection between $K'$ and
any set of representatives of $N_N(P)/C_N(P).$ Indeed,  if
$x_{\varphi}\mapsto \varphi$ denotes the mentioned bijection we have
$$kN_N(P)=\bigoplus_{\varphi\in K'}kC_N(P)x_{\varphi},$$
which is an $N_N(P)$-interior algebra such that
$$kC_N(P)x_{\varphi}\subseteq A^{\Delta_{\varphi}(P)},$$ for any $x_{\varphi}.$
Now, denoting by $\rho_{\varphi}^{K'}$ the restriction on
$kC_N(P)x_{\varphi}$ of the homomorphism of modules
$$\Br_{\Delta_{\varphi}(P)}^A:A^{\Delta_{\varphi}(P)}\rightarrow A(\Delta_{\varphi}(P)),$$
we see that, by \cite[Proposition 2.12]{Dad},
$\rho^{K'}:=\oplus_{\varphi\in K'}\rho_{\varphi}^{K'}$ is  an
$N_G(P)$-algebra isomorphism.

\end{proof}
\section{The extended Brauer quotient correspondences}\label{sec3}
We begin this section with the following theorem which may be
regarded as an extension of  \cite[Proposition 3.3]{PuZh}.
\begin{thm}\label{thmprop33} With the above notation we assume that $P\leq N_G^K(P)$ and that $P\cap N\neq \{1\}$. If $A$ is an $N$-interior $p$-permutation $G$-algebra which is projective viewed as $k(N\times 1)$-module and as
$k(1\times N)$-module, we have
$$\overline{N}_A^K(P)^P\subset A(P)+J(k(P\cap N))\overline{N}_A^K(P).$$
Moreover, for any subgroup $H$ of $G$ that contains $P$ there is a
bijective correspondence between the set of pointed groups
$H_{\beta}$ on $A$ with defect group $P$ and the set of pointed
groups $N_H(P)_{\widehat{\beta}}$ on $\overline{N}_A^K(P)$ with
defect group $P$ such that  $\Br_P(\beta)\subset \widehat{\beta}$.
Further, if $P$ is normal in $H,$  setting $$\mathcal{Q}=\{Q\leq
H\mid Q \mbox{ is a $p$-subgroup such that  } P\leq Q \},$$ then the
above bijection restricts to a bijection that preserves the defect
groups in $\mathcal{Q}.$
\end{thm}
\begin{proof} Similar arguments as the ones   proving \cite[(3.1.1)]{PuZh} give
$$J(k(P\cap N))\overline{N}_A^K(P)=\overline{N}_A^K(P)J(k(P\cap N))\subset J(\overline{N}_A^K(P)).$$
Following the same steps as in the proof of  \cite[Proposition
3.3]{PuZh} we obtain

\[\overline{N}_A^K(P)^H=A(P)^H\bigoplus \left(\bigoplus_{\varphi\in K\setminus \{\mathrm{id}_P\}}J(k(P\cap N))\overline{N}_A^{\varphi}(P)\right)^H,\]
for any subgroup $H$ contained in $N_G(P).$ So that any point of $H$
on $\overline{N}_A^K(P)$ is actually the
$(\overline{N}_A^K(P)^H)^{\times}$-conjugacy class of a primitive
idempotent belonging to $A(P)^H.$ Further, if $P\leq Q\leq H\leq
N_G(P),$ the inclusion
\[A(P)_Q^H\subseteq \overline{N}_A^K(P)^H_Q=A(P)^H_Q\bigoplus \left(\bigoplus_{\varphi\in K\setminus \{\mathrm{id}_P\}}J(k(P\cap N))\overline{N}_A^{\varphi}(P)\right)^H_Q\]
shows that the above bijection preserves the defect group $Q.$

Now let $H$ denote any subgroup of $G$ containing $P.$ Since
$$\Br_P(A^H_P)=A(P)^{N_H(P)}_P,$$ by lifting
idempotents and  the above bijection, the first assertion follows.
Next, if $P$ is normal in $H$, for any $p$-subgroup $Q$ which
verifies $P\leq Q\leq H$ we have
$$\Br_P(A^H_Q)=A(P)^{H}_Q,$$ thus the second assertion holds.
\end{proof}

 We denote by $\bar{P}$ the group $PN/N$ and define the group
 \[N_G(\bar{P}):=\{g\in G\mid g\bar{P}g^{-1}=\bar{P}\},\] due to
 the  action of $G$ on the quotient $G/N.$

\begin{thm}\label{thm}
Let $A$ be an $N$-interior  permutation $G$-algebra having an
$N\times N$-stable basis as $k$-module such that it is projective
viewed as $k(N\times 1)$-module and as $k(1\times N)$-module. Let
$H$ be any subgroup of $G$ such that $P\leq H$ and set
$D:=A\otimes_N N_G(\bar{P}).$ Then there is a bijective
correspondence preserving $P$ as defect group between the points of
$H$ on $C$ and the points of $N_H(P)$ on
$\bar{N}^{K}_{D}(P).$ Further, if $P$ is normal in $H$ then
the above bijection preserves the defect groups in
$$\mathcal{Q}=\{Q\leq H\mid Q
\mbox{ is a $p$-subgroup such that  } P\leq Q \}.$$
\end{thm}

\begin{proof} Let $B$ denote a basis of $A$ which is $G$-stable and $N\times N$-stable.
It is easy to verify that the set
\[ \{b\otimes x\mid b\in B  \mbox{ and } x\in [N_G(\bar{P})/N]\}\]
is a $N_G(\bar{P})\times N_G(\bar{P})$-stable basis of $D.$
If $\tilde{a}\in N^{\varphi}_{D}(P)$ is in a $P\times
P$-stable basis, we have $\tilde{a}=\sum b\otimes x,$ for some $b$
in $B$ and $x\in [N_G(\bar{P})/N].$  If $R$ is the subgroup
of $P$ fixing $\tilde{a},$ for any $u\in R$ we have
\[\tilde{a}={}^u\tilde{a}=u\varphi(u)^{-1}\tilde{a},\]
which is equivalent to
\[b\otimes x=u\varphi(u)^{-1} b\otimes x,\]
for any pair $b\otimes x$ appearing in the decomposition of
$\tilde{a},$ thus $\varphi(u)=u.$ Also note that
$C(P)=D(P)$ as $N_N^K(P)$-interior
$N_G\overline{}(P)$-algebra, see \cite[(9.5.4)]{Pu}. It follows that
$\Tr_R^P(\tilde{a})$ belongs to $J(N^K_{D}(P))$ if $\varphi$
is not the unity element of $K.$ At this point, the surjective
homomorphism
$$N^K_{D}(P)\to \bar{N}^K_{D}(P)$$
of permutation $N_N^K(P)$-interior $N_G(P)$-algebras forces
$$\left(\bigoplus_{\varphi\in K\setminus \{\mathrm{id_P}\}}\bar{N}^{\varphi}_{D}(P)\right)^P=
\left(\bigoplus_{\varphi\in K\setminus \{\mathrm{id_P}\}} J(k(N\cap
P))\bar{N}^{\varphi}_{D}(P)\right)^P\subseteq
J(\bar{N}^K_{D}(P)).$$ To complete the proof is suffices to
follow the steps of \cite[Proposition 3.3]{PuZh} and of Theorem
\ref{thmprop33}.
\end{proof}
\begin{rem} Let $A=kN$ as an  $N$-interior $G$-algebra and let $H=G$. In this case Theorem \ref{thm} gives a bijective correspondence between blocks of $kG$ with defect groups $P$ and primitive idempotents of $kN_G^K(P)^{N_G(P)}$ with defect groups $P$. Moreover since $kN_G^K(P)^{N_G(P)}$ is a subalgebra in the center of $kN_G(P)$, and this inclusion preserves defect groups we obtain the First Main Theorem of Brauer.
\end{rem}
Let $\Aut_G(P)$ be the group of interior automorphisms of $P$
induced by conjugation with elements from $G$. We denote by $T$ the
following subgroup
$$T:=\{\varphi\in\Aut_G(P)\mid \varphi(u)u^{-1}\in N,\forall u\in P\}.$$
We define the following $N_G^T(P)$-interior algebra (and also
$N_G(P)$-algebra in this case)
$$E:=A\otimes_{N}N_G^T(P)N,$$
with the multiplication given as before Theorem \ref{thm}.
The extended Brauer quotient of this algebra can be given more explicitly.
\begin{prop}\label{thmAhat} With the above notations we have
$$\overline{N}_{E}^T(P)=\overline{N}_A^T(P)\otimes_{N_N^T(P)}N_G^T(P),$$
as $N^T_G(P)$-interior $N_G(P)$-algebras.
\end{prop}
\begin{proof}
Let $\varphi\in T$. Then there is $z^{-1}\in N_G^T(P)$ such that
$\varphi(u)={}^{z^{-1}}u$ for any $u\in P$. First we characterize
elements from $N_{E}^{\varphi}(P)$. Let $a\otimes x$ be a
homogenous element from $E$ which is in
$N_{E}^{\varphi}(P)$ where $x\in N_G^T(P).$  The definition of
$E$ as $N_G^T(P)$-interior algebra assures us that the next
statement is true
\begin{equation}\label{eq1}
{}^{\varphi(u)}a\otimes \varphi(u)xu^{-1}=a\otimes x.
\end{equation}
We denote by $u_1$ the element $xux^{-1}\in P$, hence
$xu^{-1}=u_1^{-1}x$. Denote by $\psi\in T$ the automorphism of $P$
induced by the conjugation with $xz$. It is trivial to check that
$\varphi(u)=\psi^{-1}(u_1)$. With the above notations, from
(\ref{eq1}) we obtain
\begin{equation}\label{eq2}{}^{\varphi(u)}a\otimes \varphi(u)xu^{-1}={}^{\varphi(u)}a\otimes \varphi(u)u_1^{-1}x={}^{\psi^{-1}(u_1)}a\otimes \psi^{-1}(u_1)u_1^{-1}x=a\otimes x.
\end{equation}
The last equality is equivalent to ${}^{\psi^{-1}(u_1)}a
\psi^{-1}(u_1)u_1^{-1}=a$, that is $a\in N_A^{\psi^{-1}}(P)$. Let
$c_{x^{-1}}\in\Aut_G(P)$ be the automorphism induced by the
conjugation with $x^{-1}\in N_G^T(P)$. With this notation we have
$\psi^{-1}=\varphi\circ c_{x^{-1}}$. From the above the following
equalities are true
\begin{align*}\bigoplus_{\varphi\in T}N_{E}^{\varphi}(P)&=\bigoplus_{\varphi\in T}\left(\bigoplus_{x\in[N_G^T(P)/N_N^T(P)]}N_A^{\varphi\circ c_{x^{-1}}}(P)\otimes x\right)\\
&=\bigoplus_{x\in[N_G^T(P)/N_N^T(P)]}\left(\bigoplus_{\varphi\in T}N_A^{\varphi\circ c_{x^{-1}}}(P)\otimes x\right)\\
&=\bigoplus_{x\in [N_G^T(P)/N_N^T(P)]}N_A^T(P)\otimes x.
\end{align*}
To complete the proof it is enough to prove that $\Tr_{\Delta_{\varphi}(Q)}^{\Delta_{\varphi}(P)}(a\otimes x)=\Tr_{\Delta_{\varphi}(x^{-1}Qx)}^{\Delta_{\varphi}(P)}(a)\otimes x$ for any $Q<P$, since this forces the requested equality. We have
\begin{align*}\Tr_{\Delta_{\varphi}(Q)}^{\Delta_{\varphi}(P)}(a\otimes x)&=\sum_{(\varphi(u),u)\in[\Delta_{\varphi}(P)/ \Delta_{\varphi}(Q)]}{}^{\varphi(u)}a\otimes\varphi(u)u^{-1}~{}^ux\\
&=\sum_{(\psi^{-1}(u_1), u_1)\in[\Delta_{\varphi}(P)/ \Delta_{\varphi}(x^{-1}Qx)]}{}^{\psi^{-1}(u_1)}a\otimes\psi^{-1}(u_1)u_1^{-1}x\\
&=\sum_{(\psi^{-1}(u_1), u_1)\in[\Delta_{\varphi}(P)/ \Delta_{\varphi}(x^{-1}Qx)]}{}^{\psi^{-1}(u_1)}a\psi^{-1}(u_1)u_1^{-1}\otimes x\\
&=\Tr_{\Delta_{\varphi}(x^{-1}Qx)}^{\Delta_{\varphi}(P)}(a)\otimes
x.
\end{align*}
For the second equality we used the first part of (\ref{eq2}).
\end{proof}






\end{document}